\documentclass{amsart}
\usepackage{amssymb}
\usepackage{latexsym}
\usepackage{amsmath}
\usepackage{euscript}
\usepackage{color}
\usepackage{graphics}

\usepackage[leqno]{amsmath}

\usepackage{bbm}

\usepackage{dsfont}

\usepackage{graphicx}

\def\f{\varphi}

      \def\dC{{\mathbb C}}

   \def\dN{{\mathbb N}}   
      \def\dR{{\mathbb R}}

\newcommand{\be}{\begin{equation}}
\newcommand{\ee}{\end{equation}}
\newcommand{\ba}{\begin{eqnarray}}
\newcommand{\ea}{\end{eqnarray}}
\newcommand{\baa}{\begin{eqnarray*}}
\newcommand{\eaa}{\end{eqnarray*}}
\newcommand{\bb}{\begin{thebibliography}}
\newcommand{\eb}{\end{thebibliography}}

\newcounter{my}
\newcommand{\he}%
   {\stepcounter{equation}\setcounter{my}%
   {\value{equation}}\setcounter{equation}0%
   }%
\newcommand{\she}%
   {\setcounter{equation}{\value{my}}%
    }%

\def\dist{\operatorname{dist}}

\def\diag{{\rm diag\,}}

\newtheorem{theorem}{Theorem}[section]
\newtheorem{proposition}[theorem]{Proposition}
\newtheorem{corollary}[theorem]{Corollary}

\theoremstyle{definition}

\newtheorem{remark}[theorem]{Remark}

\numberwithin{equation}{section}

\begin{document}

\title[Jacobi matrices generated by ratios of hypergeometric functions]
{Jacobi matrices generated by ratios of hypergeometric functions}

\author{Maxim Derevyagin}

\address{
Maxim Derevyagin\\
University of Mississippi\\
Department of Mathematics\\
Hume Hall 305 \\
P. O. Box 1848 \\
University, MS 38677-1848, USA }
\email{derevyagin.m@gmail.com}


\begin{abstract}
A problem of determining zeroes of the Gauss hypergeometric function goes back to Klein, Hurwitz, and Van Vleck. In this very short note we show how ratios of hypergeometric functions arise as $m$-functions of Jacobi matrices and we then revisit the problem based on the recent developments of the spectral theory of non-Hermitian Jacobi matrices.
\end{abstract}


\maketitle

\section{Introduction}

Let us recall~\cite{AAR} that the Gauss hypergeometric function $F(a,b,c;z)$ is a special function defined by the series
\begin{equation}\label{HypDef}
F(a,b,c;z):={_{2}F_1}\left(\left.\!\!\begin{array}{c}a,b\\
c\end{array}\right|z\!\right)=1+\frac{ab}{c}\frac{z}{1!}+\frac{a(a+1)b(b+1)}{c(c+1)}\frac{z^2}{2!}+\dots,
\end{equation}
where the parameters $a$, $b$ and $c$ are complex numbers. 
Clearly, for the series~\eqref{HypDef} to be well-defined we also have to assume that $c$ is not a nonpositive integer.

Also, it is not so hard to see that if either $a$ or $b$ is a negative integer then 
$F(a,b,c;z)$ is just a polynomial. Otherwise, one can easily check by the ratio test that 
the radius of convergence of the series~\eqref{HypDef} is 1. Thus, to be more precise, by
$F(a,b,c;z)$ we understand the function defined by the series~\eqref{HypDef} for $|z|<1$ and
by analytic continuation elsewhere. 

It could sometimes be convenient to use continued fractions for understanding analytic continuations.
In order to get them in this context we will need 
the following contigous relations \cite[Section 2.5]{AAR} (see also~\cite[Section 6.1]{JTh})
\begin{equation}\label{HypId1}
F(a,b,c;z)=F(a,b+1,c+1;z)-\frac{a(c-b)}{c(c+1)}zF(a+1,b+1,c+2;z)
\end{equation}
and
\begin{equation}\label{HypId2}
F(a,b+1,c+1;z)=F(a+1,b+1,c+2;z)-\frac{(b+1)(c-a+1)}{(c+1)(c+2)}zF(a+1,b+2,c+3;z),
\end{equation}
where the latter is simply another version of \eqref{HypId1} and  \eqref{HypId1} is easily verified from~\eqref{HypDef}. Next, representing~\eqref{HypId1} in the form 
\[
\frac{F(a,b,c;z)}{F(a,b+1,c+1;z)}=1+\frac{\displaystyle{-\frac{a(c-b)}{c(c+1)}z}}
{\displaystyle{\frac{F(a,b+1,c+1;z)}{F(a+1,b+1,c+2;z)}}}
\]
and rewriting~\eqref{HypId2} in the following manner
\[
\frac{F(a,b+1,c+1;z)}{F(a+1,b+1,c+2;z)}=1+\frac{\displaystyle{-\frac{(b+1)(c-a+1)}{(c+1)(c+2)}z}}
{\displaystyle{\frac{F(a+1,b+1,c+2;z)}{F(a+1,b+2,c+3;z)}}}
\]
lead to the continued fraction~\cite[Section 6.1]{JTh}
\begin{equation}\label{CFforRofHFs}
\frac{F(a,b,c;z)}{F(a,b+1,c+1;z)}\sim
1+\displaystyle{\frac{c_1z}{1+\displaystyle{\frac{c_2z}{1+\ddots}}}}=
1+
\frac{c_1z}{1}
\begin{array}{l} \\ +\end{array}
\frac{c_2z}{1}
\begin{array}{l} \\ +\end{array}
\begin{array}{l} \\ \cdots \end{array},
\end{equation}
where
\begin{equation}\label{cformula}
  \begin{aligned}
c_{2j+1}&=-\frac{(a+j)(c-b+j)}{(c+2j)(c+2j+1)},\quad j=0,1,2\dots\\
c_{2j}&=-\frac{(b+j)(c-a+j)}{(c+2j-1)(c+2j)},\quad j=1,2,3\dots.
  \end{aligned}
\end{equation}
Next, according to~\cite[Theorem 6.1]{JTh} the continued fraction~\eqref{CFforRofHFs} converges to the meromorphic function
\[ 
\frac{F(a,b,c;z)}{F(a,b+1,c+1;z)}
\]
uniformly on compact subsets of
\[
\{w\in\dC\setminus[1,\infty): {F(a,b,c;w)}/{F(a,b+1,c+1;w)}\ne \infty\}.
\]
Actually, in a sense the main object of this note is the continued fraction \eqref{CFforRofHFs} and we will proceed in the following way. In the next section we will discuss tridiagonal matrices associated with the even part of \eqref{CFforRofHFs} for the most general case of the parameters $a$, $b$, and $c$. Then, in Section 3 we will restrict ourselves to the case of real parameters and, hence, will be able to extract more information about the underlying tridiagonal matrices or, equivalently, about ratios of hypergeometric functions.

\section{The underlying Jacobi matrices}

In this section we are going to associate Jacobi matrices with ratios of hypergeometric functions and to do that we will make a few transformations of \eqref{CFforRofHFs} at first. To begin with, let us note that the substitution $z\mapsto -1/z$ reduces 
 \eqref{CFforRofHFs} to
\begin{equation}\label{SfracH}
\frac{F(a,b,c;-1/z)}{F(a,b+1,c+1;-1/z)}=1+
\frac{d_1}{z}
\begin{array}{l} \\ +\end{array}
\frac{d_2}{1}
\begin{array}{l} \\ +\end{array}
\frac{d_3}{z}
\begin{array}{l} \\ +\end{array}
\begin{array}{l} \\ \cdots \end{array},
\end{equation}
where we set $d_j=-c_j$. Next, the continued fraction~\eqref{SfracH} inherits its convergence from \eqref{CFforRofHFs}. Therefore, the even part of the continued fraction~\eqref{SfracH} represents the same function and we have 
\begin{equation}\label{EvenP}
\frac{F(a,b,c;-1/z)}{F(a,b+1,c+1;-1/z)}=
1+
\frac{d_1}{z+d_2}
\begin{array}{l} \\ -\end{array}
\frac{d_3d_4}{z+d_3+d_4}
\begin{array}{l} \\ -\end{array}
\frac{d_5d_6}{z+d_5+d_6}
\begin{array}{l} \\ -\end{array}
\begin{array}{l} \\ \cdots \end{array},
\end{equation}
where the right-hand side, that is, the the even part of~\eqref{SfracH} is found by applying \cite[Theorem 2.10]{JTh}. Next, one can rewrite \eqref{EvenP} in the following manner
\begin{equation*}
\frac{F(a,b,c;-1/z)}{F(a,b+1,c+1;-1/z)}-1=
\frac{4d_1}{4z+4d_2}
\begin{array}{l} \\ -\end{array}
\frac{16d_3d_4}{4z+4d_3+4d_4}
\begin{array}{l} \\ -\end{array}
\frac{16d_5d_6}{4z+4d_5+4d_6}
\begin{array}{l} \\ -\end{array}
\begin{array}{l} \\ \cdots \end{array}
\end{equation*}
or, equivalently, setting $z\mapsto z/4$ we arrive at
\begin{equation}\label{EvenP1}
\frac{F(a,b,c;-4/z)}{F(a,b+1,c+1;-4/z)}-1=
\frac{4d_1}{z+4d_2}
\begin{array}{l} \\ -\end{array}
\frac{16d_3d_4}{z+4d_3+4d_4}
\begin{array}{l} \\ -\end{array}
\frac{16d_5d_6}{z+4d_5+4d_6}
\begin{array}{l} \\ -\end{array}
\begin{array}{l} \\ \cdots \end{array}.
\end{equation}
Now, introducing the function
\begin{equation}\label{Bfunction}
B(a,b,c;z)=\frac{-1}{4d_1}\left(\frac{F(a,b,c;-4/(z-2))}{F(a,b+1,c+1;-4/(z-2))}-1\right)
\end{equation}
and the coefficients $a_0=2-4d_1$,
\begin{equation}\label{abformulas}
a_n=2-4d_{2n+1}-4d_{2n+2}, \quad b_{n-1}^2=16d_{2n+1}d_{2n+2}, \quad n=1,2, \dots,
\end{equation}
we see that $B$ admits the following representation
\begin{equation}\label{Brepr}
B(a,b,c;z)=
\frac{-1}{z-a_0}
\begin{array}{l} \\ -\end{array}
\frac{b_0^2}{z-a_1}
\begin{array}{l} \\ -\end{array}
\frac{b_1^2}{z-a_2}
\begin{array}{l} \\ -\end{array}
\begin{array}{l} \\ \cdots \end{array}.
\end{equation}

\begin{proposition}
Suppose $a$, $b$, and $c$ are complex numbers such that $c$ is not a nonpositive integer. Then for the entries of the $J$-fraction representation \eqref{Brepr} of the function $B$ defined by \eqref{Bfunction} we have that
\begin{equation}\label{ablimit}
\lim_{k\to\infty}a_k=0, \qquad \lim_{k\to\infty}b_k^2=1.
\end{equation}
\end{proposition}

\begin{proof}
Since $d_j=-c_j$, the combination of \eqref{cformula} and \eqref{abformulas} immediately yields \eqref{ablimit}.
\end{proof}
The second relation in \eqref{ablimit} suggests that for sufficiently large $k$ we can use the principal square root to determine $b_k$ 
from 
\[
b_{k}^2=16\frac{(a+k+1)(c-b+k+1)(b+k+1)(c-a+k+1)}{(c+2k+2)^2(c+2k+3)(c+2k+1)}.
\]
 Therefore, we still have
\[
\lim_{k\to\infty}a_k=0, \qquad \lim_{k\to\infty}b_k=1.
\]
Furthermore, it is not so hard to check that
\begin{equation}\label{trace}
\sum_{k=0}^{\infty}|a_k|+|b_k-1|<\infty.
\end{equation}
Next, it is standard how one can associate the $J$-fraction \eqref{Brepr} with the Jacobi matrix (for details, see \cite{AptKV}, \cite{B01}, \cite{Wall48}) 
\[
J=\left(%
\begin{array}{cccc}
  {a}_0 & {b}_0 &  &  \\
 {b}_0 & {a}_1 & {b}_1 &  \\
      & {b}_1 & {a}_2 & \ddots \\
      &     & \ddots & \ddots \\
\end{array}%
\right),
\]
which generates a bounded linear operator on the Hilbert space $\ell_2$ of complex square-summable sequences equipped with the usual inner product $\left<\cdot,\cdot\right>_{\ell_2}$. Let $\rho(J)$ denote the resolvent set of $J$.
We will also need a special notation  for the free Jacobi matrix
\[
J_0=\left(%
\begin{array}{cccc}
  0 & 1 &  &  \\
  1 & 0 & 1&  \\
      & 1 & 0 & \ddots \\
      &     & \ddots & \ddots \\
\end{array}%
\right).
\]
\begin{theorem}\label{ThWB}
Suppose $a$, $b$, and $c$ are complex numbers such that $c$ is not a nonpositive integer. Then the function $B$ admits the following representation 
\begin{equation}\label{BWf}
B(a,b,c;z)=\left<(J-zI)^{-1}e,e\right>_{\ell_2}, \quad z\in\rho(J),
\end{equation}
where $J$ is the corresponding complex Jacobi matrix, $I$ is the identity operator, and $e=(1,0,0,0,\dots)^{\top}$.
Moreover, $J-J_0$ is trace class.
\end{theorem}
\begin{proof}
According to \cite[Theorem 26.2]{Wall48}, the $J$-fraction \eqref{Brepr} converges to $B$ locally uniformly in some neighborhood of infinity (it also follows from the convergence of the regular $C$-fraction \eqref{CFforRofHFs} mentioned before). At the same time, due to \cite[Corollary 4.6 (a)]{B01}, the $J$-fraction converges to the $m$-function 
$\left<(J-zI)^{-1}e,e\right>_{\ell_2}$ and, hence, formula \eqref{BWf} holds true in a neighborhood of infinity, which can be extended to $\rho(J)$ by the uniqueness of analytic functions. Finally, the fact that $J-J_0$ is trace class is immediate from \eqref{trace}. 
\end{proof}

As a result, we arrive at the following statement.
\begin{corollary}\label{C_poles} Suppose $a$, $b$, and $c$ are complex numbers such that $c$ is not a nonpositive integer. Then the function $B$ is meromorphic in $\dC\setminus[-2,2]$. Moreover, a number $\lambda\in\dC\setminus[-2,2]$ is a pole of $B$ if and only if $\lambda$ is an eigenvalue of $J$.
\end{corollary}
\begin{proof}
To prove this statement we use \eqref{BWf} and the spectral properties of $J$. At first, in view of \cite[Theorem XIII.14]{RS4} one deduces that the essential spectrum $\sigma_{ess}(J)$ of $J$ is the same as $\sigma_{ess}(J_0)$. Hence, $\sigma_{ess}(J)=[-2,2]$, which shows that $B$ is meromorphic in $\dC\setminus[-2,2]$. The second part of the statement is a consequence of \cite[Theorem 2.14]{B01}. 
\end{proof}

Due to the recent development of the field of complex Jacobi matrices it is also possible to get more information about the behavior of the poles of $B$ from the results obatined in \cite{EG05_1}, \cite{GK07}, and \cite{H11} (see also \cite{EG05} and \cite{GK12} for some generalizations). In particular, we have the following result.

\begin{corollary}
Let $\lambda_0$, $\lambda_1$, \dots be the poles of $B$ that lie outside of $[-2,2]$ and let each be listed as many times as its multiplicity. Then we have 
\[
\sum_{j=0}^{\infty}\dist(\lambda_j,[-2,2])\le \Vert J-J_0\Vert_1<\infty,
\]
where $\Vert\cdot\Vert_1$ is the trace class norm.
\end{corollary}
\begin{proof}
The statement directly follows from \cite[Theorem 2.1]{H11}, Theorem \ref{ThWB}, and Corollary \ref{C_poles}.
\end{proof}

\section{The case of the real coefficients $a$, $b$, and $c$}

Throughout this section we assume that the numbers $a$, $b$, and $c$ are real. In this case, we show that the number of non-real poles and zeroes of $B$ is finite and we give an estimate for that number. Clearly, this question is closely related to the problem of determining the number of zeroes of the Gauss hypergeometric function $F$ in $\dC\setminus[1,\infty)$ and it goes back to Klein \cite{Klein}, Hurwitz \cite{Hurwitz91}, and Van Vleck \cite{VanVleck}. Although the problem for the Gauss hypergeometric function $F$ in the case of real $a$, $b$, and $c$ was completely closed by Runckel \cite{R71}, we propose a different approach to get some information about zeroes and poles of $B$. Besides, the approach, which is based on the theory of continued fractions, the generalized Jacobi matrices, and generalized Nevanlinna functions (see \cite{D03}, \cite{DD04}, \cite{DD07}), allows us to see the general structure of $B$. 

At first, let us make the following observation. 

\begin{proposition}[cf. the second paragraph on page 341 of \cite{Wall48}]\label{StabR}
Let $a$, $b$, and $c$ be real numbers such that $c$ is not a nonpositive integer and let $a_j$ and $b_j$ be the entries of the $J$-fraction representation \eqref{Brepr} of $B(a,b,c;z)$. Then $a_j$ is a real number for 
$j=0, 1,2,\dots$ and 
\[
b_j^2>0
\]
for sufficiently large integers $j$.
\end{proposition}
\begin{proof}
Note that if $a$, $b$, and $c$ are real numbers and $c$ is not a nonpositive integer, then the coefficients $c_k$ defined by \eqref{cformula} are real for $k=1,2,3,\dots$. Furthermore, one can easily see from \eqref{cformula} that $c_k$ are negative for sufficiently large $k$. Thus, the statement follows from \eqref{abformulas} and the fact that $d_j=-c_j$.
\end{proof}

Before going into details of the general real case, let us quickly consider the classical case, which occurs when each $c_k<0$. Namely, the condition
\begin{equation}\label{CCforR}
0<a<c+1,\quad 0<b+1<c+1, \quad c>0
\end{equation}
guarantees that $c_k<0$ for $k=1,2,3,\dots$ and, in this case, the continued fraction \eqref{CFforRofHFs} represents a Stieltjes function, which can be stated in the following way.

\begin{proposition} Under the condition \eqref{CCforR} we have 
\begin{equation}\label{Bnev}
B(a,b,c;z)=\int_{-2}^{2}\frac{d\mu(a,b,c;t)}{t-z},
\end{equation}
where $\mu(a,b,c;t)$ is a positive probability measure on $[-2,2]$.
\end{proposition} 
\begin{proof}
Notice that 
\[
\frac{1}{z}\left(\frac{F(a,b,c;z)}{F(a,b+1,c+1;z)}-1\right)\sim
\frac{c_1}{1}
\begin{array}{l} \\ +\end{array}
\frac{c_2z}{1}
\begin{array}{l} \\ +\end{array}
\begin{array}{l} \\ \cdots \end{array}.
\]
Then, applying \cite[formula (89.14)]{Wall48} and \cite[formula (67.5)]{Wall48} to the above continued fraction
yields that
\[
\frac{1}{z}\left(\frac{F(a,b,c;z)}{F(a,b+1,c+1;z)}-1\right)=\int_0^1\frac{d\theta(u)}{1-zu},
\]
which, after appropriate simplifications and substitutions, leads to the desired representation \eqref{Bnev}. 
\end{proof}
\begin{remark} The explicit formula for $\mu(a,b,c;t)$ can be extracted from the findings of \cite{Bel82}. Also, it is worth mentioning that the closed formula for the approximants to the corresponding continued fraction can be found in \cite{Wimp}.
\end{remark}

Formula \eqref{Bnev} shows that $B$ is a Nevanlinna function provided that the condition \eqref{CCforR} is satisfied. 
Furthermore, it turns out that to study properties of $B$ in the general real case it is natural to invoke the theory of generalized Nevanlinna functions, which include Nevanlinna functions as a proper subclass. At first, recall that 
generalized Nevanlinna functions were introduced by M.G. Krein and H. Langer and some information about them can be found in~\cite{KL77}. To give a precise definition of generalized Nevanlinna functions let us consider a function $\varphi$ that is meromorphic on ${\dC}\setminus{\dR}$ and that satisfies  the  symmetry  condition  $\overline{\varphi(z)}=\varphi({\overline{z}})$. Also, let $\rho(\f)$ denote the domain of holomorphicity of $\varphi$. Then, 
for a nonnegative integer $\kappa$, the
generalized Nevanlinna class  ${\bf N}_{\kappa}$
consists of functions $\varphi$ such  that  the kernel
\[
\begin{cases}
       {N}_{\f}(z,w)=
\frac{\f (z)-\overline{\f(w)}}{z-\overline{w}},& \text{$z,w\in\rho(\f)$},\\
   {N}_{\f}(z,\overline{z})=\f'(z),& \text{$z\in\rho(\f)$}\\
  \end{cases}
\]
has $\kappa$ negative squares, which means that for all choices of
$p\in{\dN}$ and $z_{1},z_{2},\dots,z_{p}\in\rho(\f)$ the Hermitian matrix
\[
\left({N}_{\f}(z_{i},z_{j})\right)_{i,j=1}^p
\] 
has at most $\kappa$ and for at least one such choice exactly $\kappa$ negative eigenvalues.
 
Clearly, ${\bf N}_{0}$ coincides with the class of Nevanlinna functions that map the upper half-plane $\dC_+$ into the upper half-plane. It is well known \cite[Chapter 3, Section 1]{A65} that a classical Nevanlinna function $\f$ admits the following integral representation
\[
\varphi(z)=\nu_1 z+\nu_2+\int_{\dR}\frac{1+tz}{t-z}d\tau(t),
\]
where $\nu_1>0$, $\nu_2$ is a real number, and $\tau$ is a non-decreasing function of bounded variation. Moreover, if $\varphi$
is holomorphic in a neighborhood of infinity and verifies the condition
\[
\varphi(z)=O\left(\frac{1}{z}\right), \quad z\to\infty
\]
then it has the following representation 
\[
\varphi(z)=\int_{\alpha}^{\beta}\frac{d\sigma(t)}{t-z},
\]
where $\sigma$ is a nonnegative measure $\sigma$ on $[\alpha,\beta]$. 

Unfortunately, the integral representation of generalized Nevanlinna functions is complicated. However, to understand the structure of generalized Nevanlinna functions one may use a factorization result from \cite{DLLS00}. Namely, if $\f\in{\bf N}_{\kappa}$ then 
there exist numbers $\alpha_1$, $\alpha_2$, \dots, $\alpha_{\kappa_1}\in\dC_+\cup\dR$ and $\beta_1$, $\beta_2$, \dots, $\beta_{\kappa_2}\in\dC_+\cup\dR$
such that
\begin{equation}\label{NkF}
\varphi(z)=\frac{\displaystyle{\prod_{j=1}^{\kappa_1}(z-\alpha_j)(z-\bar{\alpha}_j)}}{\displaystyle{\prod_{j=1}^{\kappa_2}(z-\beta_j)(z-\bar{\beta}_j)}}\widehat{\varphi}(z),
\end{equation}
where $\kappa_1$, $\kappa_2\le \kappa$ and $\widehat{\varphi}\in{\bf N}_0$ is a classical Nevanlinna function.

Finally, we are in the position to relate generalized Nevanlinna functions to our previous discussion. Before doing that, let us notice here that in view of Corollary \ref{C_poles} the function $B$ is holomorphic at a neighborhood of infinity and equals zero at infinity. Hence, we are interested in functions that are holomorphic at some neighborhood of infinity and equal zero at infinity. That is why in what follows we assume that 
\[
\varphi(z)=-\frac{s_0}{z}-\frac{s_1}{z^2}-\frac{s_2}{z^3}-\cdots, \quad |z|>R
\]
for sufficiently large number $R>0$. So, the last piece is the following particular case of the algorithm elaborated in \cite{D03} and used for developing the accompanying theory of generalized Jacobi matrices in \cite{DD04} and \cite{DD07}.  
\begin{proposition}\label{InSchur}
Let $\kappa$ be a nonnegative integer and let $\psi\in{\bf N}_{\kappa}$. Define a function $\f$ by the formula
\[
\f(z)=-\frac{\varepsilon}{z-\gamma+\varepsilon \delta^2\psi(z)},
\] 
where $\varepsilon=\pm 1$, $\gamma$ is a real number and $\delta>0$. Then we have the following:

\begin{enumerate}
\item[(i)] If $\varepsilon=1$ then $\f\in{\bf N}_{\kappa}$,
\item[(ii)] If $\varepsilon=-1$ then $\f\in{\bf N}_{\kappa+1}$.
\end{enumerate} 
\end{proposition}
\begin{proof}
The statement is a particular case of \cite[Theorem 3.2]{D03}. Alternatively, the proof of the statement can easily be extracted from the reasoning given in \cite[Section 2.3]{DD07}. 
\end{proof}

To formulate the main result of this section we need to introduce a special sequence $\varepsilon_j=\pm 1$. To do that let us recall that according to Proposition \ref{StabR} there is a nonnegative integer $N$ such that 
\[
b_{j}^2>0, \quad j=N, N+1, N+2, \dots
\]
and $b_{N-1}^2<0$ is the last negative number in the sequence $\{b_j^2\}_{j=0}^{\infty}$. Obviously, one can pick a sequence of numbers $\varepsilon_j=\pm 1$, where $j=0,1,2,\dots$, such that 
\[
\widetilde{b}_0^2=\varepsilon_0\varepsilon_1b_{0}^2>0,\quad
\widetilde{b}_1^2=\varepsilon_1\varepsilon_2b_{1}^2>0, \quad
\widetilde{b}_2^2=\varepsilon_2\varepsilon_3b_{2}^2>0, \quad \dots
\]
and $\varepsilon_N=1$, $\varepsilon_{N+1}=1$, $\varepsilon_{N+2}=1$, \dots. To be definite here, it should be stressed that we choose $\widetilde{b}_j$ to be positive for each index $j$.
\begin{theorem}
Let $a$, $b$, and $c$ be real numbers such that $c$ is not a nonpositive integer and let $\kappa$ be the number of $-1$'s in the sequence $\varepsilon_0$, $\varepsilon_1$, \dots, $\varepsilon_{N-1}$. Then $\varepsilon_0B\in{\bf N}_{\kappa}$. In other words, there exist numbers $\alpha=\alpha(a,b,c)$, $\beta=\beta(a,b,c)\in\dR$, $\alpha_1=\alpha_1(a,b,c)$, \dots, $\alpha_{\kappa_1}=\alpha_{\kappa_1}(a,b,c)\in\dC_+\cup\dR$ and $\beta_1=\beta_1(a,b,c)$, \dots, $\beta_{\kappa_1}=\beta_{\kappa_1}(a,b,c)\in\dC_+\cup\dR$
such that
\begin{equation}\label{NkFB}
B(a,b,c;z)=\varepsilon_0\frac{\displaystyle{\prod_{j=1}^{\kappa_1}(z-\alpha_j)(z-\bar{\alpha}_j)}}{\displaystyle{\prod_{j=1}^{\kappa_2}(z-\beta_j)(z-\bar{\beta}_j)}}\left(\nu_1 z+\nu_2+\int_{\alpha}^{\beta}\frac{d\mu(a,b,c;t)}{t-z}\right),
\end{equation}
where $\kappa_1$, $\kappa_2\le \kappa$, $\nu_1=\nu_1(a,b,c)>0$, $\nu_2=\nu_2(a,b,c)$ is a real number, and
$\mu(a,b,c;t)$ is a positive finite measure on $[\alpha,\beta]\supseteq [-2,2]$.
\end{theorem}
\begin{proof}
At first, note that to get \eqref{NkFB} from the fact that $\varepsilon_0B\in{\bf N}_{\kappa}$ is easy. Indeed, one just needs to apply formula \eqref{NkF} and use the fact that $B$ is holomorphic at infinity. So, the essential part of the proof is to see that $\varepsilon_0B\in{\bf N}_{\kappa}$, which is done by consecutive applications of Proposition \ref{InSchur}. More precisely, let us define a function $\varphi_N$ in the following way
\[
\varphi_N(z)=\frac{-1}{z-a_N}
\begin{array}{l} \\ -\end{array}
\frac{b_N^2}{z-a_{N+1}}
\begin{array}{l} \\ -\end{array}
\frac{b_{N+1}^2}{z-a_{N+2}}
\begin{array}{l} \\ -\end{array}
\begin{array}{l} \\ \cdots \end{array},
\]
which can be rewritten as 
\[
\varphi_N(z)=\frac{-\varepsilon_N}{z-a_N}
\begin{array}{l} \\ -\end{array}
\frac{b_N^2}{z-a_{N+1}}
\begin{array}{l} \\ -\end{array}
\frac{b_{N+1}^2}{z-a_{N+2}}
\begin{array}{l} \\ -\end{array}
\begin{array}{l} \\ \cdots \end{array}
\]
since $\varepsilon_N=1$. Then, due to \cite[Theorem 66.1]{Wall48} and the fact that $b_j^2>0$ for $j=N, N+1, N+2,\, \dots$ we know that $\varphi_N$ is a Nevanlinna function, that is, $\varphi_N\in{\bf N}_0$. The next step is to define another function $\varphi_{N-1}$ via the relation
\[
\varphi_{N-1}(z)=-\frac{\varepsilon_{N-1}}{z-a_{N-1}+\varepsilon_{N-1}\widetilde{b}_{N-1}^2\varphi_N(z)}.
\]
By the construction, we have $\varepsilon_{N-1}=-1$ and, hence, Proposition \ref{InSchur} yields that $\varphi_{N-1}\in{\bf N}_1$. Repeating this procedure $N-1$ times we get that $\varphi_0=\varepsilon_0B\in{\bf N}_{\kappa}$, which completes the proof.
\end{proof}

To conclude this section and the paper, it is worth mentioning that, as is shown in \cite{DD04} and \cite{DD07}, in the real case it is natural to deal with tridiagonal matrices with real entries rather than with symmetric complex Jacobi matrices as it was done before for the most general case. Namely, in the real case one can consider the following tridiagonal matrix   
\[
H=\begin{pmatrix}
  a_o & \widetilde{b}_0&  &\\
  \varepsilon_0\varepsilon_1\widetilde{b}_0& a_1& \widetilde{b}_1&\\
     & \varepsilon_1\varepsilon_2\widetilde{b}_1&a_2&\\
     &&&\ddots&    
      \end{pmatrix},      
\]
where $\varepsilon_j=\pm 1$ and, more importantly, $\varepsilon_j=1$ for $j=N, N+1, N+2,\, \dots$. 
Consequently, the matrix $H$, which is not symmetric in general, is a finite rank perturbation of a real symmetric Jacobi matrix. This type of tridiagonal matrices is a very particular case of the generalized Jacobi matrices introduced and studied in \cite{DD04} and \cite{DD07}. In the entire generality, the generalized Jacobi matrices play the same role for generalized Nevanlinna functions as Jacobi matrices do for Nevanlinna functions. Also, the bounded generalized Jacobi matrices, which is the case for us, lead to self-adjoint and bounded operators in Pontryagin spaces. To quickly demonstrate it, let us define the diagonal matrix $G=\diag(\varepsilon_0,\varepsilon_1,\varepsilon_2, \dots)$. Then,
if we consider the bilinear form on $\ell_2$
\[
(x,y)_G=(Gx,y)_{\ell_2}, \quad x,y\in \ell_2,
\]  
we see that $(Hx,y)_G=(x,Hy)_G$. Next, following \cite{DD04} and \cite{DD07} we can also introduce the $m$-function of $H$ via the formula
\[
m(z)=((H-z)^{-1}e,e)_G, \quad e=(1,0,0,\dots)^{\top},
\]
which is proved to be a generalized Nevanlinna function. Remarkably, some nonclassical orthogonal polynomials on the unit disk were introduced in \cite{DS17} and the Szeg\H{o} mapping applied to those polynomials leads to tridiagonal matrices that have the same structure as $H$ does \cite{DS17sum}.

\end{document}